\documentclass[a4paper,11pt,twoside]{article}

\usepackage{mathrsfs}

\usepackage{mathtools}
\usepackage{comment}

\usepackage{listings}
\makeatletter
\newcommand{\julgreen}{\textcolor[rgb]{0.22,0.60,0.15}}
\newcommand{\julblue}{\textcolor[rgb]{0.25,0.39,0.85}}
\newcommand{\julred}{\textcolor[rgb]{0.80,0.24,0.20}}
\makeatother
\lstdefinelanguage{Julia}{
  morekeywords={function,return,end,for,while,if,else,elseif,break,continue},
  morekeywords=[2]{}, 
  sensitive=true,
  morecomment=[l]\#,
  morestring=[b]"
}
\usepackage{hyperref}
\hypersetup{
 pdftitle={Practical Guide},
 pdfauthor={LAMOHPL},
 colorlinks=true,
 linkcolor=blue,
 citecolor=blue,
 filecolor=blue,
 urlcolor=blue}

 \usepackage{coupled80}

 \usepackage{tcolorbox}            

\begin{document}

\title{A Practical Guide to Rigorously Locate \\ Periodic Orbits in Discrete Dynamics}
\titlerunning{Computer-assisted Proofs}

\author{Lucia Alonso Mozo$^{1,*}$
\quad Olivier H\'enot$^{2}$
\quad Phillipo Lappicy$^3$}
\authorrunning{L. A. Mozo, O. H\'enot, P. Lappicy}
\address{ 
$^{1}$Universidad Complutense de Madrid, Spain. \quad  \texttt{lucia.alonso.mozo@gmail.com} \\ 
$^{*}$This author is now at Google Zurich, Switzerland. \\ 
$^{2}$National Taiwan University, Taiwan. \quad\texttt{olivierhenot@ntu.edu.tw}\\
$^{3}$Universidad Complutense de Madrid, Spain. \quad\texttt{philemos@ucm.es}
}

\email{}

        \abstract{Periodic orbits are important objects of discrete dynamical systems, but finding them is not always easy. We present a self-contained introductory account, aimed at non-experts, to prove their existence and study their stability using the aid of the computer. The method consists in three main steps. 
        First, we reformulate the problem of identifying a $p$-periodic orbit as a root-finding problem. 
        Second, we find a numerical approximation of the root (i.e. a periodic orbit candidate). 
        Third, we verify rigorously the contraction of a quasi-Newton operator near this approximation, which guarantees the existence of a unique root (i.e. periodic orbit). The neighbourhood of contraction is a ball centered at the approximation, whose radius yields a rigorous a posteriori error bound on the numerical approximation. 
        \newline
        To illustrate the effectiveness of this method, we implement it in two examples: the well-known logistic map and a discretization of a predator prey model.
        For the logistic map, we prove the existence of more than $80\cdot 10^2$ periodic orbits of periods $p=1,\ldots , 80$, mostly unstable. 
        For the predator-prey model, we rigorously detect over $80\cdot 10^4$ periodic orbits of periods $p=1,\ldots, 10$, mostly unstable as well. 
        %
        %
        This confirms well-known dynamical features such as period-doubling bifurcations and the emergence of increasingly complex orbit structures as the parameter changes. 
        %
        %
}

\keywords{Discrete dynamical systems, periodic orbits, stability, computer-assisted proofs.}

\msc{37E05; 37G15; 37N25; 65G30; 92D25.}


\makeppage

\vskip .25em%
  
\tableofcontents

\pagestyle{plain}

\newpage


\section{Introduction}

Numerical analysis has thrived in the development of algorithms that provide accurate and efficient approximations for problems in linear and nonlinear algebra, differential equations, optimization, quadrature, and more.
More broadly across mathematics, the computer has become a valuable research companion.
What is sometimes overlooked, however, is that computers can further assist in the construction of rigorous mathematical proofs. 
This paper aims to offer an accessible introduction to computer-assisted proofs in discrete dynamical systems, specifically for models that arise in the context of mathematical biology.

Computer-assisted proofs 
encompass wide spectrum of approaches, ranging from formal proofs to validated numerics.
What concerns us here are \emph{a posteriori validation methods}, in which a given numerical simulation is rigorously verified.
With the seminal work of Lanford on the Feigenbaum conjecture in the 80's,~\cite{Lanford}, it further developed with the work of Nakao in Japan, Plum in Germany, Mischaikow in the United-States, Mrozek and Zgliczy\'nski in Poland, and more recently Lessard in Canada; see \cite{Mischaikow95,Piotr97,Lessard16,NakaoPlum}.
Some notable achievements pertaining to dynamical systems include the resolution of Smale's 14th problem \cite{tucker2002rigorous}, the existence of transient chaos in the Kuramoto-Sivashinksky PDE \cite{Wilczak}, the proof of Jones' conjecture \cite{Jaquette}, and the proof of Marchal's conjecture \cite{Marchal}.
See also the survey articles~\cite{van2015rigorous,MirelesMischaikow,gomez2019computer}.


We consider the discrete dynamical system defined by the iteration of an initial state $x_0 \in \mathbb{R}^n$ under a diffeomorphism $f : \mathbb{R}^n \to \mathbb{R}^n$,
\begin{equation} \label{eq:sistema_general}
x_k = f(x_{k-1}) = f^k (x_0), \qquad k = 1, 2, \dots.
\end{equation}
The map $f$ may stem directly from a discrete-time model, or arise through the Poincar\'e map of a differential equation.
Invariant sets, and in particular periodic orbits, serve as anchors in the phase space, around which the global dynamics is organized.
However, finding them analytically is generally difficult, see for example~\cite[Ex. 1.4.4]{guckenheimer2013nonlinear}.
We develop both a theoretical framework and an algorithm to prove the existence of stable and unstable periodic orbits of any given period, even in settings that display chaotic behavior.
In addition, we introduce a uniform contraction argument to rigorously trace the curves of period-doubling bifurcations, which separate regions of period-$p$ and period-$2p$ in parameter space.
To illustrate the approach, we apply these techniques to two case studies: the logistic map, as a toy model, and a predator-prey system.
In both cases, we detected periodic orbits of various lengths, from period $p = 2$ to $p = 80$, with a high numerical precision (i.e., with error margins of at least $10^{-11}$).

All the files used to obtain the results presented in the upcoming sections can be found in the GitHub repository available at \cite{githubGitHubLuciaalonsomozoCAPs}.

\subsection{A short note on interval arithmetic}
\label{sec:num_representables}

The standard arithmetic in computers relies on \emph{binary floating-point arithmetic}, see~\cite{goldberg1991what}.
Barring precision limitations, a (reduced) rational number $p/q$ has an exact binary representation if and only if its denominator $q$ is a power of $2$.
For example, $1/4$ is exactly $0.01_2$ in base-2, while $1/10$ has the infinite repeating representation $0.0001100110011\ldots_2$.
Even then, working in finite precision induces rounding errors; numerical results for ill-conditioned problems, or computations sensitive to small perturbations, are particularly prone to deviate from the true solution.

In the 60's, Moore addressed these constraints by introducing the \emph{interval arithmetic} \cite{Moore66}.
Today there is an extensive literature on this topic, see \cite[Chapter 2]{Tucker11} and \cite{IEEE1788}.
Yet, in the interest of preserving the clarity of what follows, let us outline the central ideas.
The approach is to forgo the manipulation of exact mathematical numbers, and instead work with the set of intervals:
\begin{equation}
\mathbb{IF} \coloneqq \left\{ [\underline{a}, \overline{a}] \subset \mathbb{R} \, : \, \underline{a} \le \overline{a}, \, \underline{a}, \overline{a} \in \mathbb{F} \right\},
\end{equation}
where $\mathbb{F} \subset \mathbb{R}$ denotes the set of floating-point numbers.
%
The precision determining $\mathbb{F}$ controls how tightly a real number $\mathbb{R}$ can be represented in $\mathbb{IF}$.
Letting $\nabla, \Delta : \mathbb{R} \to \mathbb{F}$ respectively denote the rounding down and up operations, then we have that $\pi \in [\nabla(\pi), \Delta(\pi)]$, for example.
A rigorous enclosure of arithmetic operations $+, -, \times, \div$ is then defined over $\mathbb{IF}$, e.g. $[\underline{a}, \overline{a}] +_\mathbb{IF} [\underline{b}, \overline{b}] \coloneqq [\nabla(\underline{a} + \underline{b}), \Delta(\overline{a} + \overline{b})]$.
Further definitions exist for elementary functions and other special functions, e.g. the exponential, the logarithm, trigonometric functions, Bessel function, gamma function.

There are many software libraries currently available in various programming language, such as \texttt{IntervalArithmetic} (Julia), \texttt{INTLAB} (Matlab), and \texttt{MPFI} (C++) to name a few.


\section{Existence, local uniqueness and stability}
\label{sec:cap}

Our goal is to lay out a practical framework to ``locate'' periodic orbits and establish their stability.
Here, ``locate'' is understood as ``determining a small neighborhood containing the true periodic solution''.
For a brief introduction on discrete dynamical system theory, see \cite[Chapter 1.4]{guckenheimer2013nonlinear}; whereas for a thorough review, see \cite{devaney2018introduction} and references therein.
For the purposes of this article, recall the following definition:
\begin{definition}
We say that $x^\star_0 \in \mathbb{R}^n$ is a periodic point of period $p \ge 1$ if $x^\star_0 = f^p(x^\star_0)$ and $f^k(x^\star_0) \ne f^j(x^\star_0)$ for all $k \ne j \in \{ 0, \ldots , p-1\}$.
Moreover, the $p$-periodic point $x^\star_0 \in \mathbb{R}^n$ is called: 

(i) stable if for all $\epsilon > 0$, there exists $\delta > 0$ such that $\|x_0 - x^\star_0\| < \delta$ implies that $\|f^{pk}(x_0) - x^\star_0\| < \epsilon$ for all $k\geq 1$.

(ii) asymptotically stable if it is stable and there exists $\tilde{\delta} > 0$ such that for all $x_0\in \mathbb{R}^n$ satisfying $\| x_0 - x^\star_0 \| \le \tilde{\delta}$, then $
\lim_{k \to \infty} f^{pk}(x_0) = x^\star_0$.



(iii) unstable if it is not stable, i.e.,
if there exists $\epsilon > 0$ such that for all $\delta > 0$, there is some $x_0\in \mathbb{R}^n$ satisfying $\|x_0 - x^\star_0\| < \delta$ such that $\|f^{pk}(x_0) - x^\star_0\| \geq \epsilon$ for some $k\geq 1$.
\end{definition}
To distinguish the stability of a periodic point $x^\star_0\in \mathbb{R}^n$ of period $p$, we typically use the following linear criterion based on eigenvalues $\lambda_1, \dots, \lambda_n$ of the matrix $Df^p(x^\star_0)$, see \cite[Chapter 1.4]{devaney2018introduction}. If $|\lambda_i| < 1$ for all $i = 1, \dots, n$, then $x^\star_0\in \mathbb{R}^n $ is asymptotically stable. If $|\lambda_i| > 1$ for some $i \in \{ 1, \dots, n\}$, then $x^\star\in \mathbb{R}^n $ is unstable.   
Whenever $|\lambda_i| = 1$ for some $i \in \{1, \dots, n\}$, the linear stability criterion is inconclusive.

We begin by observing that it is considerably simpler to represent a periodic point through its iterates rather than attempting to write down the $p$-fold composition of $f$ by itself, i.e., $f^p$.
%
%
Hence, we recast the problem of finding periodic points as the search for a zero of the map $F : \mathbb{R}^{np} \to \mathbb{R}^{np}$ given by
\begin{equation}\label{eq:def_F}
F(x) \coloneqq 
\begin{pmatrix*}
f(x_{p-1}) - x_0 \\
f(x_0) - x_1 \\
\vdots \\
f(x_{p-2}) - x_{p-1}
\end{pmatrix*},
\end{equation}
where $x = (x_0,\ldots, x_{p-1})\in \mathbb{R}^{np}$, such that each component $x_k \in \mathbb{R}^{n}$ for $k=0,\ldots, p-1$.
%
%
Therefore, if $F(x^\star)=0$ for some $x^\star=(x^\star_0,\ldots, x^\star_{p-1}) \in  \mathbb{R}^{np}$ and $x^\star_k\neq x^\star_j$ for all $k\neq j \in \{0,\ldots, p-1\}$, then $x^\star_0\in\mathbb{R}^n$ yields a $p$-periodic point orbit of the dynamical system~\eqref{eq:sistema_general}.
This motivates our abuse of notation, in which we denote by $x_k\in \mathbb{R}^{n}$ both the iterates of the discrete dynamical system in~\eqref{eq:sistema_general} and the $k$-th coordinate of the variable $x = (x_0, \dots, x_{p-1}) \in \mathbb{R}^{np}$.

Our strategy consists in obtaining a ``good enough'' numerical approximation $\bar{x} = (\bar{x}_0, \dots, \bar{x}_{p-1})$ of the periodic orbit (i.e., $F(\bar{x}) \approx 0$), and then construct an operator, contracting ``near'' $\bar{x}$, whose fixed-points are periodic orbits.
This is referred to as an \emph{a posteriori validation method}, in which notions of ``good enough'' and ``near'' will be explicitly and rigorously quantified.
Hence, we seek to prove that the following quasi-Newton operator is a contraction:
\begin{equation}\label{eq:def_T}
x \mapsto x - A F(x),
\end{equation}
where $A : \mathbb{R}^{np} \to \mathbb{R}^{np}$ is a linear operator, constructed such that
\begin{equation}\label{AapproxDF}
A \approx DF(\bar{x})^{-1}.
\end{equation}
Many numerical strategies are available to identify a numerical zero $\bar{x}$ of $F$, see \cite[Chapter 2]{Atkinson}.
However, Newton's method, defined by the iterates $y_k = y_{k-1} - DF(y_{k-1})^{-1} F(y_{k-1})$ for $k = 1, 2, \dots$, 
%
%
ought to be employed at some stage to refine $\bar{x}$ to a prescribed numerical tolerance.
Near an isolated zero of $F$, convergence of the $y_k$'s is quadratic; in fact, a numerical failure to converge quadratically might signal that the contraction operator given in \eqref{eq:def_T} is not suitable.

Assuming that $\bar{x}$ and $A \approx DF(\bar{x})^{-1}$ have been computed numerically, the following theorem provides sufficient conditions for $x \mapsto x - A F(x)$ to be a contraction in a neighborhood $\bar{x}$.
The theorem is stated without restricting ourselves to finite-dimensional problems.
\begin{notation}
Henceforth, given a Banach space $\mathcal{X}$, the closed ball of radius $r \ge 0$ and centered at $x \in \mathcal{X}$ is denoted by $B_r(x) \coloneqq \left\{ y \in \mathcal{X} \, : \, \| y - x \|_\mathcal{X} \le r \right\}$.
Moreover, the set of bounded linear operators from $\mathcal{X}$ to itself is denoted by $\mathscr{B}(\mathcal{X})$.
\end{notation}
\begin{theorem}\label{thm:contraction}
Let $\mathcal{X}$ be a Banach space, $\bar{x} \in \mathcal{X}$, $F \in \mathcal{X} \to \mathcal{X}$ a $C^1$ map, and $A : \mathcal{X} \to \mathcal{X}$ an injective linear map.
Fix $R \in [0, \infty]$ and assume the existence of $Y$, $Z_1$, and $Z_2 = Z_2(R)$ such that
\begin{subequations}\label{eq:bounds}
\begin{align}
\|AF(\bar{x})\|_\mathcal{X} &\le Y, \label{eq:boundsY}\\
\|I - ADF(\bar{x}) \|_{\mathscr{B}(\mathcal{X})} &\le Z_1, \label{eq:boundsZ1}\\
\|A(DF(x) - DF(\bar{x}))\|_{\mathscr{B}(\mathcal{X})} &\le Z_2 \| x- \bar{x}\|_\mathcal{X}, \quad \text{for all } x \in B_R(\bar{x}).\label{eq:boundsZ2}
\end{align}
\end{subequations}
If there exists $r \in [0, R]$ such that
\begin{subequations}\label{eq:cond}
\begin{align}
P(r) \coloneqq Y + r (Z_1 - 1) + \frac{r^2}{2} Z_2 &\le 0\\
Z_1 + r Z_2 &< 1,\label{eq:cond_Z1}
\end{align}
\end{subequations}
%
then there exists a unique $x^\star \in B_r (\bar{x})$ such that $F(x^\star) = 0$.
\end{theorem}
\begin{proof}
See, for example \cite[Theorem 1.2.1]{Breden25}, for a complete proof.
\end{proof}
\begin{remark}\label{rem:rpa}
We emphasize three points in regard to \Cref{thm:contraction}:
\begin{enumerate}
\item \textbf{Interpretation of $Y$, $Z_1$, $Z_2$, $R$.}
Each bound in \eqref{eq:bounds} should be read as follows:
\begin{itemize}  
\item $Y$ measures how close $\bar{x}$ is to the true zero $x^\star$ of $F$.
\item $Z_1$ measures how accurately $A$ approximates $DF(\bar{x})^{-1}$.
\item $Z_2$ measures the local variation (Lipschitz constant) of $x \mapsto A D F(x)$ in $B_R (\bar{x})$.
\item $R$ is an fixed radius, viewed as an a priori error threshold on $\bar{x}$.
In practice, this is often chosen ad hoc by means of heuristics, depending in particular on $Y$.
\end{itemize}
Any radius $r$ satisfying the inequalities \eqref{eq:cond} gives a rigorous error bound on $\bar{x}$; the smallest such radius yields the tightest error bound, while the largest yields the biggest ball within which \Cref{thm:contraction} guarantees the uniqueness of the zero $x^\star$ of $F$.

\item \textbf{Choice of $A$.}
In finite dimensions, it is possible to take $A = DF(\bar{x})^{-1}$, where we use interval arithmetic to enclose $DF(\bar{x})^{-1}$ so that $A$ is stored as a matrix with interval entries.
Such a choice would give $Z_1 = 0$, which is optimal.
However, in practice this is generally a marginal gain compared to the computational cost of inverting interval matrices.

\item \textbf{Injectivity of $A$.}
The injectivity of $A$ may not need to be verified a priori; it can follow from the computation of $Z_1$ satisfying \eqref{eq:boundsZ1} and the condition~\eqref{eq:cond_Z1}.
Specifically, if $Z_1 < 1$, then $A DF(\bar{x})$ is invertible, which implies that $A$ is surjective.
In finite dimensions, $\dim \mathcal{X} < \infty$, we always have that surjectivity implies injectivity, and so verifying the condition~\eqref{eq:cond_Z1} implies that $A$ is injective.
In infinite dimensions, whether surjectivity is enough to guarantee injectivity depends on the structure of $A$, though ensuring this property is typically pursued.
\end{enumerate}
\end{remark}

\begin{remark}\label{rem:APP}
Applying \Cref{thm:contraction} can be summarized in the following steps:
\begin{enumerate}
\item Formulate the problem (e.g., in the context of this section, finding periodic points of a discrete dynamical system) as a zero-finding problem of a map $F : \mathcal{X} \to \mathcal{X}$.
\item Compute \emph{numerically} (e.g., using Newton's method) an approximation $\bar{x}\in \mathcal{X}$ of a zero of $F$ and an approximation $A$ of the inverse of $DF(\bar{x})$.
\item Compute \emph{rigorously} (e.g., using interval arithmetic) the bounds $Y$, $Z_1$, $Z_2$, and the roots $r_\pm$ of $P(r)$ and, if these are real, check that there exists $r \in [r_-, r_+] \cap [0, R]$ such that the inequalities \eqref{eq:cond} are satisfied.
\end{enumerate}
\end{remark}

For the existence of periodic points, the map $F$ is given in \eqref{eq:def_F} and the space for the contraction is $\mathcal{X} = \mathbb{R}^{np}$ endowed with the $1$-norm.
Specifically, for any $x = (x_1, \dots, x_{np}) \in \mathbb{R}^{np}$ and $A\in \mathscr{B}(\mathbb{R}^{np})$, we have
\begin{equation}\label{eq:ell1}
\|x\|_\mathcal{X} \coloneqq \sum_{i = 1}^{np} |x_i|, \qquad \|A\|_{\mathscr{B}(\mathcal{X})} \coloneqq \max_{1\le j\le np} \sum_{i = 1}^{np} |A_{ij}|.
\end{equation}
Finally, it remains to check that $x^\star \in \mathbb{R}^{np}$ yields a genuine period-$p$ orbit of $f$, i.e., not an orbit of smaller period.
This is the case provided that all the $p$-coordinates of $x^\star = (x^\star_0,\ldots, x^\star_{p-1}) \in \mathbb{R}^{np}$ are distinct.
Since $x^\star_k \in B_r(\bar{x}_k)\subset \mathbb{R}^n$ for $k = 0, \dots, p-1$, 
we must show that the neighborhoods of all $\bar{x}_k$ with the a posteriori error bound $r>0$ are disjoint, i.e.,
\begin{equation}
\bigcap_{k=0}^{p-1}  B_r(\bar{x}_k) = \emptyset.
\end{equation}
%
%
%
%
%
%
Now, assuming that a period-$p$ orbit $x_0^\star$ of $f$ has been determined, we seek to retrieve the eigenvalues of the Jacobian of $f^p (x^\star_0)$.
The chain rule yields
\begin{align}
Df^{p}(x_0^\star)
&= Df(f^{p-1}(x_0^\star)) \cdot Df(f^{p-2}(x_0^\star)) \cdots Df(x_0^\star) \nonumber \\
&= Df(x_{p-1}^\star) \cdot Df(x_{p-2}^\star) \cdots Df(x_0^\star).
\label{eq:EV}
\end{align}
For $n=1$, which is the case for the examples we consider in the upcoming sections, the unique eigenvalue is given by the value $\lambda = Df^p(x_0^\star)$ and it can be computed \emph{rigorously} (e.g., using interval arithmetic) through~\eqref{eq:EV}.
For $n>1$, to capture the spectrum $\sigma(Df^p(x_0^\star))=\{\lambda_1,\dots,\lambda_n\}$, one can use the Gershgorin circle theorem to encapsulate the eigenvalues within appropriate disks.
%

\subsection{Example: The logistic map}\label{subsec:logistic}

This first application is a step-by-step example.
Alongside each formula, we include a code snippet to indicate how the implementation of the computer-assisted proof looks like.
We follow Julia's syntax, the programming language used for all of our computations and proofs; see \cite{githubGitHubLuciaalonsomozoCAPs}.

The logistic map

\begin{minipage}[c]{0.45\linewidth}
\begin{equation}\label{logisticmap}
f_\mu (x) \coloneqq \mu x (1 - x),
\end{equation}
\end{minipage}%
\hfill%
\begin{minipage}[c]{0.45\linewidth}
\begin{tcolorbox}[colback=white, colframe=black!40, boxrule=0.5pt, arc=2mm, boxsep=0mm]
\begin{lstlisting}[language=Julia, xleftmargin=0pt, xrightmargin=0pt, aboveskip=0pt, belowskip=0pt]
function f(x, mu)
    return mu * x * (1 - x)
end
\end{lstlisting}
\end{tcolorbox}
\end{minipage}

where $\mu \in [0,4]$ and $x \in [0,1]$, was first studied by Lorenz \cite{lorenz1964problem} and later by May \cite{may1976simple}; it illustrates how varying $\mu$ can lead to complex dynamics 
as $\mu$ approaches $4$.
See~\cite{Stewart} for a brief overview of the logistic map, whereas~\cite[Chapter 4]{SymmInChaos} and \cite{devaney2018introduction} provide a broad discussion of its dynamics. 
%

Let us prove the existence of a stable period-2 point for $\mu = 3.2$.
This point can be found explicitly, see~\cite{lorenz1964problem}; we choose it here for the sake of simplicity, to clearly illustrate the general case of a period-$p$ point.
As it is needed for our proof, the derivative of the logistic map reads

\begin{minipage}[c]{0.45\linewidth}
\begin{equation}
f_\mu'(x) = \frac{\mathrm{d}}{\mathrm{d}x} f_\mu(x) = \mu (1 - 2x).
\end{equation}
\end{minipage}
\hfill
\begin{minipage}[c]{0.45\linewidth}
\begin{tcolorbox}[colback=white, colframe=black!40, boxrule=0.5pt, arc=2mm, boxsep=0mm]
\begin{lstlisting}[language=Julia, xleftmargin=0pt, xrightmargin=0pt, aboveskip=0pt, belowskip=0pt]
function Df(x, mu)
    return mu * (1 - 2 * x)
end
\end{lstlisting}
\end{tcolorbox}
\end{minipage}

\textbf{Step 1 (reformulate the problem).}
We define $F : \mathbb{R}^{2} \to \mathbb{R}^{2}$ as the particular case of \eqref{eq:def_F} with $p=2$ and $n=1$, where a zero of $F$ corresponds to a period-2 point. This yields the following map and its Jacobian matrix:

\begin{minipage}[c]{0.45\linewidth}
\begin{equation}
F_\mu (x) \coloneqq
\begin{pmatrix}
f_\mu (x_1) - x_0 \\
f_\mu (x_0) - x_1 \\
\end{pmatrix},
\end{equation}
\end{minipage}
\hfill
\begin{minipage}[c]{0.5\linewidth}
\begin{tcolorbox}[colback=white, colframe=black!40, boxrule=0.5pt, arc=2mm, boxsep=0mm]
\begin{lstlisting}[language=Julia, xleftmargin=0pt, xrightmargin=0pt, aboveskip=0pt, belowskip=0pt]
function F(x, mu)
    return [f(x[2], mu) - x[1]
            f(x[1], mu) - x[2]]
end
\end{lstlisting}
\end{tcolorbox}
\end{minipage}
\begin{minipage}[c]{0.45\linewidth}
\begin{equation}
DF_\mu (x) =
\begin{pmatrix}
-1 & f_\mu '(x_0) \\
f_\mu '(x_1) & -1
\end{pmatrix}.
\end{equation}
\end{minipage}
\hfill
\begin{minipage}[c]{0.5\linewidth}
\begin{tcolorbox}[colback=white, colframe=black!40, boxrule=0.5pt, arc=2mm, boxsep=0mm]
\begin{lstlisting}[language=Julia, xleftmargin=0pt, xrightmargin=0pt, aboveskip=0pt, belowskip=0pt]
function DF(x, mu)
    return [-1  Df(x[2], mu)
            Df(x[1], mu)  -1]
end
\end{lstlisting}
\end{tcolorbox}
\end{minipage}

\textbf{Step 2 (numerical approximation).}
There are two quantities that are computed numerically: an approximate zero $\bar{x}$ of $F_\mu$ and an approximate Jacobian matrix $A$ of $DF_\mu (\bar{x})^{-1}$.

First, we obtain a numerical approximation $\bar{x}=(\bar{x}_0 ,\bar{x}_1)$ of the period-2 point, for some $\bar{\mu}\approx 3.2$, using Newton's method; see~\cite{githubGitHubLuciaalonsomozoCAPs} on how to implement it.
For simplicity, let us proceed with the proof using only the first two digits of the approximation:

%
\begin{minipage}[c]{0.4\linewidth}
\begin{equation}\label{eq:approx}
\begin{aligned}
\bar{\mu} &\approx 3.2, \\
\bar{x} &\approx (0.51, 0.79).
\end{aligned}
\end{equation}
\end{minipage}
\hfill
\begin{minipage}[c]{0.55\linewidth}
\begin{tcolorbox}[colback=white, colframe=black!40, boxrule=0.5pt, arc=2mm, boxsep=0mm]
\begin{lstlisting}[language=Julia, xleftmargin=0pt, xrightmargin=0pt, aboveskip=0pt, belowskip=0pt]
# Def. approx. parameter and orbit
mu_bar = 3.2
x_bar = [0.51, 0.79]
\end{lstlisting}
\end{tcolorbox}
\end{minipage}

We stress once more that 32/10, 51/100 and 79/100 cannot be represented exactly in a computer (see \Cref{sec:num_representables}), and thus the $\approx$ symbol highlights that $\bar{\mu}$, $\bar{x}_0$, $\bar{x}_1$ correspond to the closest 64-bit binary floating-point number, respectively.

Second, we obtain an approximate inverse $A$ of the Jacobian $DF_{\bar{\mu}} (\bar{x})$, as in \eqref{AapproxDF}.
In practice, the matrix $A$ is computed as a numerical approximation of $DF_{\bar{\mu}} (\bar{x})^{-1}$ via an algorithm implemented in a function \texttt{inv}, with the following outcome:

%
\begin{minipage}[c]{0.58\linewidth}
 \begin{equation}\label{logistic:A}
 A \approx \begin{pmatrix*}
 2.10618055051 & -1.1347955552327 \\
 -1.1347955552327 & 0.07262691553489
%
 \end{pmatrix*}.
 \end{equation}
\end{minipage}
\hfill
\begin{minipage}[c]{0.4\linewidth}
\begin{tcolorbox}[colback=white, colframe=black!40, boxrule=0.5pt, arc=2mm, boxsep=0mm]
\begin{lstlisting}[language=Julia, xleftmargin=0pt, xrightmargin=0pt, aboveskip=0pt, belowskip=0pt]
# Compute approx. inverse
A = inv(DF(x_bar, mu_bar))
\end{lstlisting}
\end{tcolorbox}
\end{minipage}

Note that in this case of low dimension ($n=1$) and low period ($p=2$), we can find the exact expression for $DF_\mu (0.51, 0.79)$ and $DF_\mu (0.51, 0.79)^{-1}$ for the exact parameter value $\mu=3.2$:
\begin{equation}\label{logistic:A_true}
DF_\mu (0.51, 0.79) = 
\begin{pmatrix}
-\frac{8}{125} & -1 \\
-1 & -\frac{232}{125}
\end{pmatrix},
\qquad
DF_\mu (0.51, 0.79)^{-1} = \frac{125}{13769}
\begin{pmatrix}
232 & -125 \\
-125 & 8
\end{pmatrix}.
\end{equation}
Hence, the numerical approximation in~\eqref{logistic:A} is close to the exact values in~\eqref{logistic:A_true}.
However, for larger dimensional problems, it is more effective to use a dedicated numerical algorithm \texttt{inv}.

\textbf{Step 3 (rigorous -- interval arithmetic).}
We verify the hypothesis of \Cref{thm:contraction}, where we recall that $\mathcal{X}=\mathbb{R}^p$ is equipped with the $1$-norm in \eqref{eq:ell1}.
We have that $Z_2 = 2\mu \|A\|_{\mathscr{B}(\mathcal{X})}$ from
%
%
%
\begin{equation}
\|A(DF_\mu(x) - DF_\mu(\bar{x}))\|_{\mathscr{B}(\mathcal{X})} 
\le \|A\|_{\mathscr{B}(\mathcal{X})} \|DF_\mu(x) - DF_\mu(\bar{x})\|_{\mathscr{B}(\mathcal{X})} 
= 2\mu \|A\|_{\mathscr{B}(\mathcal{X})} \|x - \bar{x}\|_\mathcal{X}.
\end{equation}
Since the map $F_\mu$ is quadratic, this inequality holds on the entire space $\mathbb{R}^p$ and hence we can choose $R = \infty$.
We can now rigorously compute the bounds $Y,Z_1,Z_2$ using interval arithmetic.
In practice, this requires that each previous numerical value must be wrapped inside an interval via a function \texttt{interval} (e.g., available in the software \texttt{IntervalArithmetic} \cite{interval}); doing so:

\begin{minipage}[c]{0.35\linewidth}
\begin{equation}\label{logistic:YZ1Z2}
\begin{aligned}
Y   &= \|A F_\mu(\bar{x})\|_\mathcal{X} \\
    &\in [0.012775, 0.0127751], \\
Z_1 &= \|I - A DF_\mu(\bar{x})\|_{\mathscr{B}(\mathcal{X})} \\
    &\in [0.0, 6.66134 \cdot 10^{-16}], \\
Z_2 &= 2 \mu \|A\|_{\mathscr{B}(\mathcal{X})} \\
    &\in [20.7422, 20.7423].
\end{aligned}
\end{equation}
\end{minipage}
\hfill
\begin{minipage}[c]{0.6\linewidth}
\begin{tcolorbox}[colback=white, colframe=black!40, boxrule=0.5pt, arc=2mm, boxsep=0mm]
\begin{lstlisting}[language=Julia, xleftmargin=0pt, xrightmargin=0pt, aboveskip=0pt, belowskip=0pt]
# Enclosure of mu = 32/10 and floats
imu = interval(32)/interval(10)
ix_bar = interval(x_bar)
iA = interval(A)
# Compute rigorously the bounds
Y = norm(iA * F(ix_bar, imu), 1)
Z_1 = opnorm(I - iA * DF(ix_bar, imu), 1)
Z_2 = 2 * imu * opnorm(iA, 1)
\end{lstlisting}
\end{tcolorbox}
\end{minipage}
%
%
%
%
The next task is to compute rigorously the roots $r_\pm$ of $P(r)$, and check that there exists some $r \in [r_-, r_+]$ satisfying both inequalities \eqref{eq:cond}: 

\begin{minipage}[c]{0.45\linewidth}
\begin{equation}\label{per2:errorroots}
r_\pm = \frac{1-Z_1 \pm \sqrt{(1-Z_1)^2 - 2YZ_2}}{Z_2}.
\end{equation}
\end{minipage}
\hfill
\begin{minipage}[c]{0.55\linewidth}
\begin{tcolorbox}[colback=white, colframe=black!40, boxrule=0.5pt, arc=2mm, boxsep=0mm]
\begin{lstlisting}[language=Julia, xleftmargin=0pt, xrightmargin=0pt, aboveskip=0pt, belowskip=0pt]
# Find the smallest root r_min
delta = (1 - Z_1)^2 - 2 Y * Z_2
r_m = (1 - Z_1 - sqrt(delta)) / Z_2
\end{lstlisting}
\end{tcolorbox}
\end{minipage}
For the sake a reporting a value expressible exactly in both base-2 and base-10, we choose
\begin{equation}\label{per2:error}
r_* \coloneqq 2^{-6} = 0.015625 \in [ r_-, r_+].
\end{equation}
%
As explained in \Cref{rem:rpa}, Point 3, the matrix $A$ in~\eqref{logistic:A} is necessarily injective since $Z_1 < 1$.
Hence, due to~\Cref{thm:contraction}, a true period-2 orbit $x^\star$ exists within a distance (in $1$-norm) $r_*$ of $\bar{x}$.

It is expected that the a posteriori error on $x^\star$, given by $r_*$ in \eqref{per2:error}, is of order $10^{-2}$, since we chose $\bar{x}$ to be accurate up to $2$ decimals.
We obtain a better error by reproducing the proof with $\bar{x}=(\bar{x}_1,\bar{x}_2) \approx (0.513044509531044, 0.7994554904683129)$, yielding an error $r_*$ of order $10^{-12}$. 

\textbf{Checking periodicity $p = 2$ (rigorous -- interval arithmetic).}
We need to ensure that $x^\star = (x_0^\star, x_1^\star)$ truly is a period-2 orbit, and not of lower order one (in this case, a fixed-point).
This is evident from the values for $\bar{x} = (\bar{x}_0, \bar{x}_1)$ in \eqref{eq:approx} and the error bound $r_*$ reported in~\eqref{per2:error}.

\textbf{Checking the stability (rigorous -- interval arithmetic).}
For $n=1$, the Jacobian in~\eqref{eq:EV} 
corresponds to the unique eigenvalue $\lambda = D f^2_\mu (x_0^\star)$, 
thus we find (using interval arithmetic):

\vspace{-0.42cm}
\begin{minipage}[c]{0.55\linewidth}
\begin{equation}\label{per2:Jac}
    \lambda = f_\mu'(x_1^\star) f_\mu'(x_0^\star) \in [-0.0644712, 0.314457],
\end{equation}
\end{minipage}
\hfill
\begin{minipage}[c]{0.25\linewidth}
\hfill 
\end{minipage}

\,
\begin{minipage}[c]{0.1\linewidth}
\hfill
\end{minipage}
\hfill
\begin{minipage}[c]{0.85\linewidth}
\begin{tcolorbox}[colback=white, colframe=black!40, boxrule=0.5pt, arc=2mm, boxsep=0mm]
\begin{lstlisting}[language=Julia, xleftmargin=0pt, xrightmargin=0pt, aboveskip=0pt, belowskip=0pt]
# Enclosure of the true periodic orbit
x0_star = interval(ix_bar[1] - r_m, ix_bar[1] + r_m)
x1_star = interval(ix_bar[2] - r_m, ix_bar[2] + r_m)
# Compute rigorously the eigenvalue
lambda = Df(x1_star, imu) * Df(x0_star, imu)
\end{lstlisting}
\end{tcolorbox}
\end{minipage}
which implies that the periodic point $x_0^\star$ is stable.

Repeating this argument for different periods, we proved the existence of 8906 unstable orbits and 372 stable orbits of periods $p = 2,\ldots, 80$. 
In particular, we prove the existence of a period-3 to obtain chaos, as in~\cite{li2004period}.
The incidence of periodic orbits, as $\mu$ increases, is consistent with the known cascade of period-doubling bifurcations, see~\Cref{fig:ec_logistica}.
See \cite{githubGitHubLuciaalonsomozoCAPs} for the code.
\begin{figure}[h!]
\centering
\includegraphics[height=1.8in]{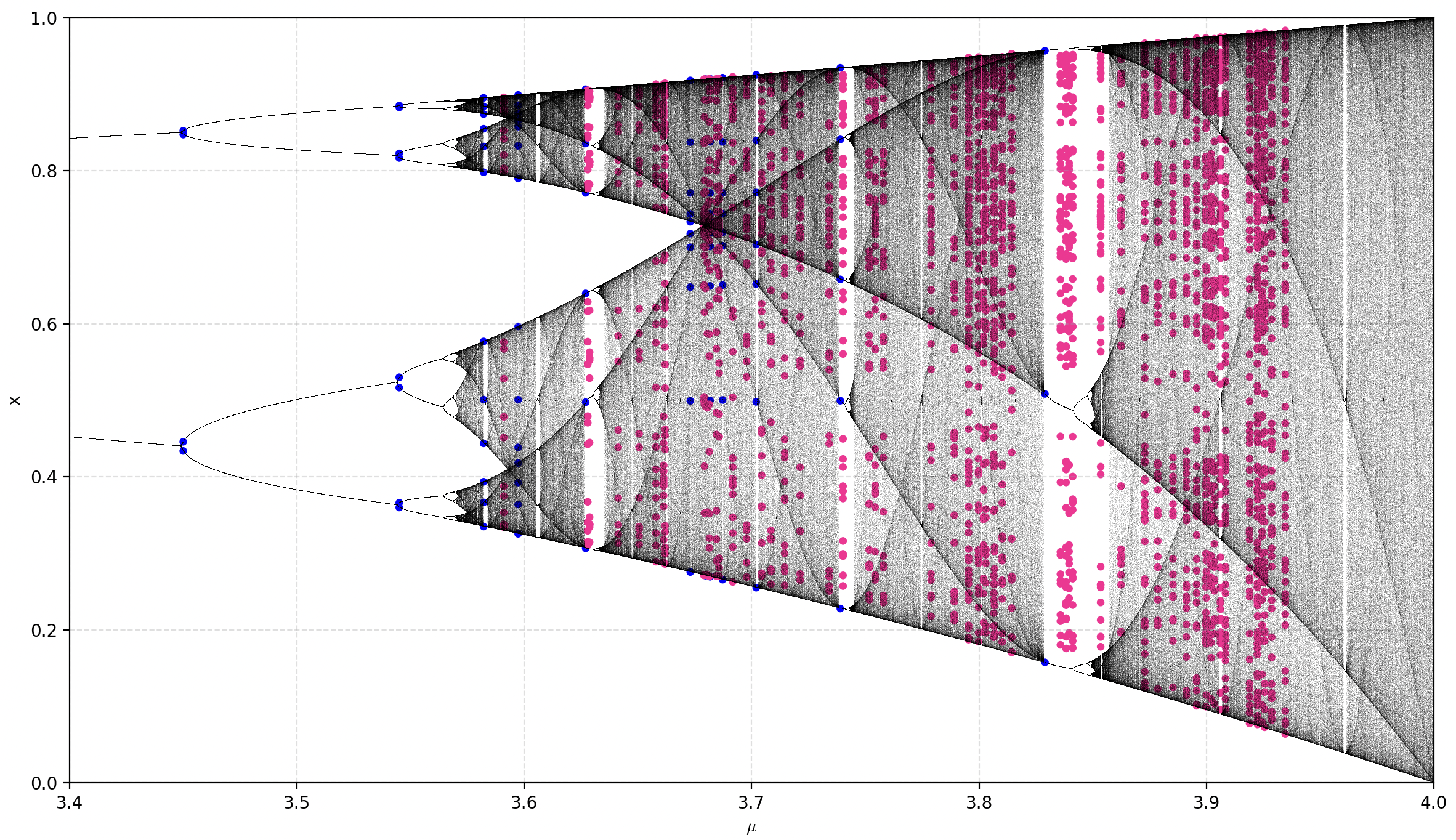}
\includegraphics[height=1.8in]{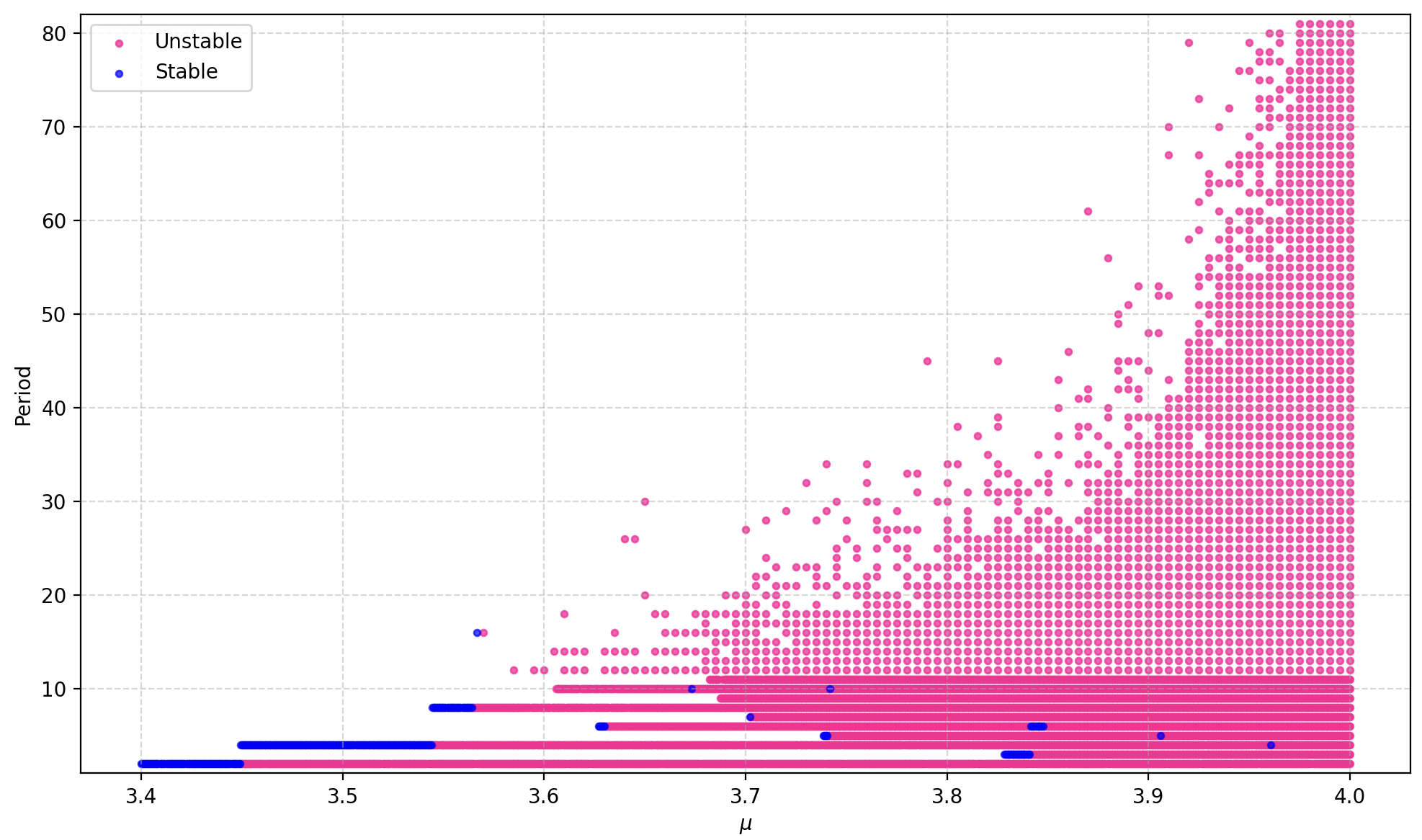}
\caption{\textbf{Left:} The well-known bifurcation diagram for the logistic map containing one period-$p$ orbit for each $p=2,\ldots,80$ described in \cite{githubGitHubLuciaalonsomozoCAPs}.
\textbf{Right:} The number of periodic points (of periods up to $p=80$) increases with $\mu$.
Blue (resp. magenta) denotes stable (resp. unstable) points.
%
}
\label{fig:ec_logistica}
\end{figure}
\subsection{Example: A predator-prey model}

We now turn to a discrete dynamical system that models certain population dynamics, as it is an approximation of 
the Poincar\'e map of an ordinary differential equation of two competing predators for one renewable prey, see~\cite{kryzhevich2021bistability}. 
The map reads 
\begin{equation}\label{eq:sistema_ejemplo1}
f(x) = f(x,\beta,\kappa) \coloneqq \beta + x - \frac{\kappa}{1 + e^x},
\end{equation}
where $x \in \mathbb{R}$ and we restrict ourselves to the parameter region $P \coloneqq \{(\beta, \kappa) \in \mathbb{R}^2 \, : \, \kappa < \beta < 0 \}$.
There is a fixed-point of $f$, given by $x_\textnormal{fp} \coloneqq \ln(\kappa/\beta - 1)$,
%
%
which undergoes a bifurcation, see~\cite[Lemma 1 and 3]{kryzhevich2021bistability}, leading to the appearance of a unique period-2 orbit for any 
\begin{equation}\label{kappa_pd}
    \kappa\leq \kappa_{\textnormal{pd}(2)}(\beta)\coloneqq\frac{\beta^2}{\beta+2} \qquad \text{ and } \qquad \beta<-2 .
\end{equation}
%
%
The authors also hint that there may exist periodic orbits of higher period, 
see~\cite[Lemma 2]{kryzhevich2021bistability}. 
In particular, they illustrate sub-regions in the parameter space associated with the \emph{numerical} existence of stable periodic orbits of higher periods, see~\cite[Figure 4]{kryzhevich2021bistability}.
%

We follow the three steps in the guide of \Cref{sec:cap} to obtain a plethora of periodic orbits for \eqref{eq:sistema_ejemplo1}: 
(1) we formulate the appropriate zero-finding problem $F$, (2) we compute the numerical root $\bar{x}$ of $F$ and inverse $A\approx DF (\bar{x})^{-1}$, and (3) we apply \Cref{thm:contraction}. For the latter, the bounds $Y$ and $Z_1$ are abstractly defined as in~\eqref{logistic:YZ1Z2}; however, the bound of $Z_2$ has to be adapted for the Lipschitz constant of \eqref{eq:sistema_ejemplo1}. Indeed, we consider $Z_2 =  (\sqrt{3} / 18) |\kappa| \, \|A\|_{\mathscr{B}(\mathcal{X})} $
and $R =\infty$, due to
\begin{equation}\label{PP:Z2}
\begin{aligned}
\|A(DF(x) - DF(\bar{x}))\|_{\mathscr{B}(\mathcal{X})} &\le \|A\|_{\mathscr{B}(\mathcal{X})} \|DF(x) - DF(\bar{x})\|_{\mathscr{B}(\mathcal{X})} \\
&\le \|A\|_{\mathscr{B}(\mathcal{X})} \left( \sup_{x\in B_R(\bar{x})} \| D^2 F(x)\|_{\mathscr{B}(\mathcal{X},\mathscr{B}(\mathcal{X}))} \right) \|x - \bar{x}\|_\mathcal{X}  \\
&\le \frac{\sqrt{3} \, |\kappa|}{18} \|A\|_{\mathscr{B}(\mathcal{X})} \|x - \bar{x}\|_\mathcal{X},
\end{aligned}
\end{equation}
where we use the mean value inequality and the fact that $|f''(x)| = |\kappa| \, e^x \, |1 - e^x| /(e^x + 1)^3$ has a maximum value given by $\sqrt{3} |\kappa| / 18$, which occurs at the points \(x = \ln(2 \pm \sqrt{3})\).
We developed an algorithm to rigorously explore the parameter plane in~\cite{githubGitHubLuciaalonsomozoCAPs}. Specifically, several values of \( b \in [-20, -2] \) and \( \kappa \in [-45, -5] \) were considered, with different step sizes, and periodic points of periods ranging from $p=2$ to $p=10$ were tested.
As a result, 820,921 unstable and 143,477 stable periodic orbits were identified and rigorously computed, 
see~\Cref{fig:ec_dp2}.

We confirm the presence of several stable orbits that have been found numerically in~\cite{kryzhevich2021bistability}, see \Cref{fig:ec_dp2} (left); in particular, the suggestion of a period-doubling and saddle-node bifurcations. 
%
%
The two regions of unstable orbits that are observed in  \Cref{fig:ec_dp2} (middle) are symmetric, see in~\cite[Remark 2]{kryzhevich2021bistability}. 
These two regions are separated by an empty band, a zone containing a stable period-2 orbits as the global attractor. 
%
%
The region with smaller $\kappa$ and larger $b$ exhibit rapid growth in the number of unstable orbits with increasing period, which may reflect a route to chaos through a cascade of period-doubling bifurcations, see \Cref{fig:ec_dp2} (right).
\begin{figure}[h!]
\centering
\includegraphics[height=1.85in]{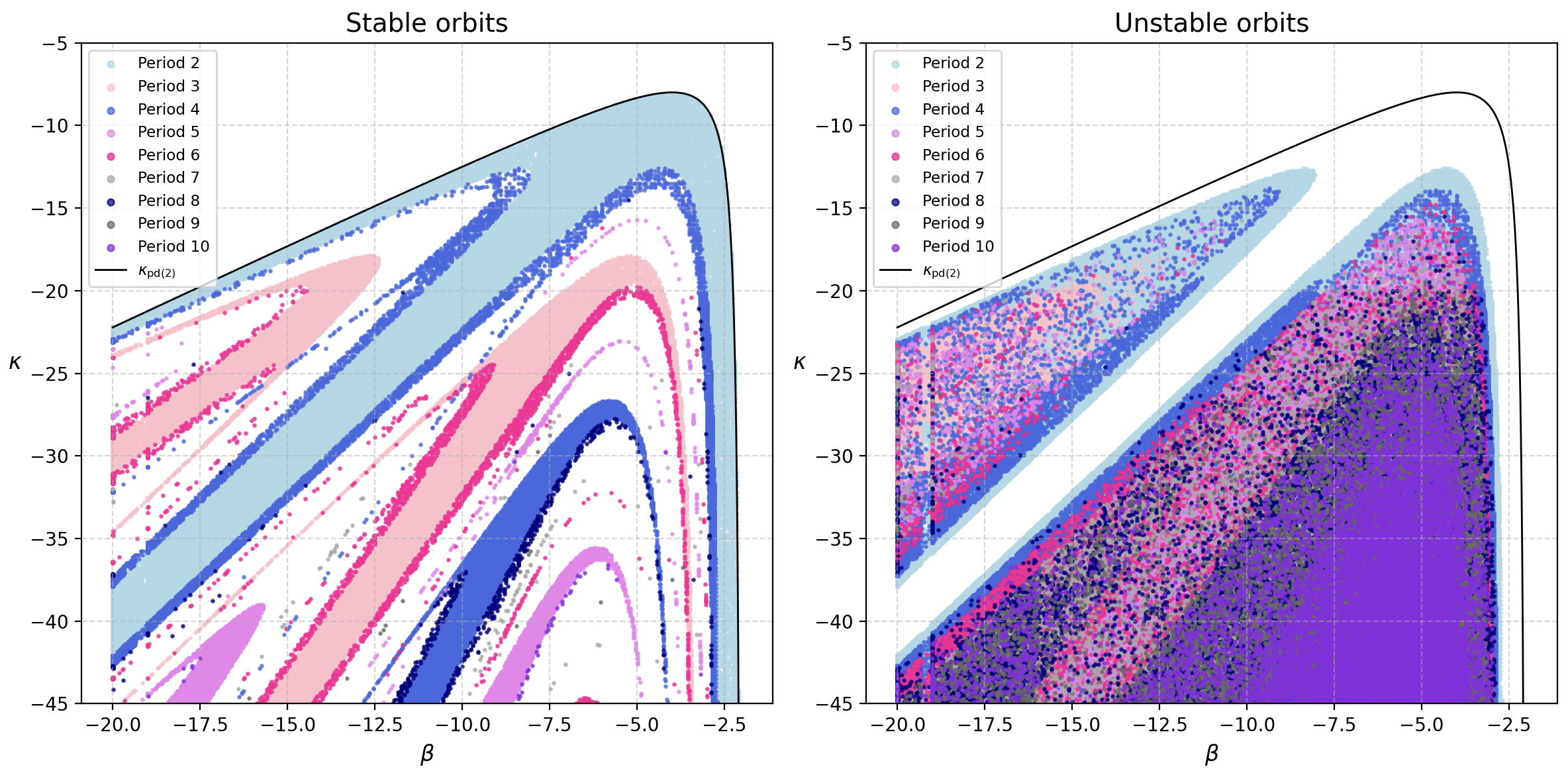}
\includegraphics[height=1.8in]{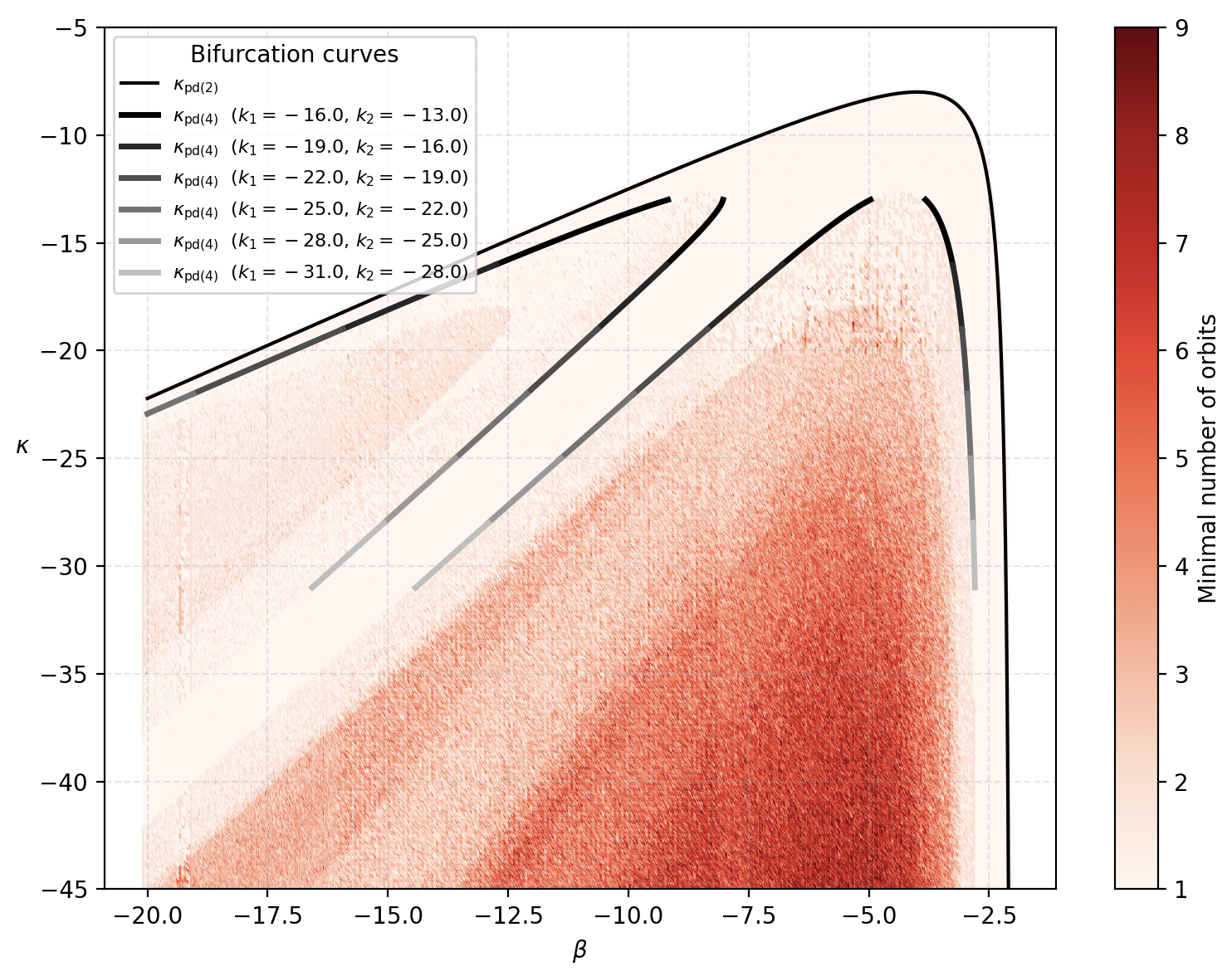}
\caption{\textbf{Left and middle:} Regions of the parameter plane \((\beta, \kappa)\) associated with stable and unstable periodic points. 
Note the curves that separate regions of period $p$ from period $2p$, associated to period-doubling bifurcations; these are \emph{rigorously} constructed in~\Cref{sec:perdob}. These curves form the so-called \emph{shrimp structures}, 
see~\cite{Gallas,Stoop,Glass,kryzhevich2021bistability}.
\textbf{Right:} Heatmap that depicts if there are periodic points of periods up to a given value \(p\) in the parameter plane \((\beta, \kappa)\). Here, 
we do not account for how many orbits of the same period may exist; we simply plot the existence of periodic points for all periods up to a fixed period $p$, which determines the shade of red. 
}
\label{fig:ec_dp2}
\end{figure}

\section{Continuum of period-doubling bifurcations}\label{sec:perdob}

%
%
In this section, we rigorously detect the boundary, in the two-parameter space $(\beta, \kappa)$ of the predator-prey model in \eqref{eq:sistema_ejemplo1}, between regions where period-$2p$ points emerge from period-$p$ points; see~\Cref{fig:ec_dp2}.
Such a phenomenon is called a \emph{period-doubling bifurcation} and a cascade of such events often leads downstream to chaotic dynamics, see~\cite{SanderYorke,devaney2018introduction}.
In our case, given a $p$-periodic point $x_0^\star$ of $f$, the bifurcation occurs when the eigenvalue $D_x f^p(x_0^\star, \beta, \kappa)$ monotonically crosses $-1$ as a parameter varies.
In particular, this means that there exists a nontrivial element $v \in \ker (D_x f^p(x_0^\star, \beta, \kappa) + 1)$; for $n=1$, we can normalize $v=1$, and thus we obtain the following condition:
\begin{equation}\label{eq:eig_prob}
Df^p(x_0^\star, \beta, \kappa) + 1 = 0.
\end{equation}
%
%
%
%
For convenience, we let $\kappa\in [\kappa_1,\kappa_2]$ be the varying parameter, for given $\kappa_1,\kappa_2\in\mathbb{R}$ with $\kappa_1<\kappa_2$; whereas we solve for $\beta^\star (\kappa) $ such that $\lambda(\kappa) = D_x f^p(x_0^\star(\kappa), \beta^*(\kappa), \kappa)$ is an eigenvalue equal to $-1$.
For reasons that will become clear later, we reparametrize the parameter $\kappa\in [\kappa_1,\kappa_2]$ as $\alpha\in [-1, 1]$ according to 
%
%
%
\begin{equation}
\kappa(\alpha) \coloneqq \kappa_1 + \frac{\alpha + 1}{2} (\kappa_2 - \kappa_1), \qquad \text{for all } \alpha \in [-1, 1], \label{kappa_eta}
\end{equation}
and thereby we consider the function with the rescaled parameter $f (x, \beta, \kappa(\alpha))$ from now on.

We now formulate our problem of finding periodic orbits and a continuum of period-doubling bifurcations as a zero-finding problem of the map $\mathcal{F} :\mathcal{W} \times [-1,1] \to \mathcal{W}$ given by
\begin{equation}\label{eq:map_period_doubling}
\mathcal{F} (w, \alpha) \coloneqq
\begin{pmatrix}
F(x, \beta, \alpha) \\
G(x, u, \beta, \alpha)
\end{pmatrix}, \qquad w = (x, u, \beta) \in \mathcal{W} \coloneqq\mathbb{R}^p \times \mathbb{R}^{p-1} \times \mathbb{R},
\end{equation}
where $F$ and $G$, in view of~\eqref{eq:def_F}, \eqref{eq:eig_prob} and \eqref{kappa_eta}, 
are defined as follows:
\begin{equation}\label{def:G}
F(x, \beta, \alpha) \coloneqq
\begin{pmatrix*}
f(x_{p-1},\beta,\kappa(\alpha)) -x_0 \\
f(x_0,\beta,\kappa(\alpha)) - x_1 \\
\vdots \\
f(x_{p-2},\beta,\kappa(\alpha)) - x_{p-1}
\end{pmatrix*},
G(x, u, \beta, \kappa) \coloneqq
\begin{pmatrix*}
D_x f(x_{p-1},\beta,\kappa(\alpha)) u_{p-2} \cdots u_0 + 1 \\
D_x f(x_0,\beta,\kappa(\alpha)) - u_0 \\
\vdots \\
D_x f(x_{p-2},\beta,\kappa(\alpha)) - u_{p-2}
\end{pmatrix*},
\end{equation}
noticing that we recast the problem in~\eqref{eq:eig_prob} by an equivalent formulation of finding a zero of $G$, by introducing a new variable $u=(u_0,\ldots,u_{p-2})\in \mathbb{R}^{p-1
}$, since $Df^p(x, \beta, \kappa)$ is given by~\eqref{eq:EV}.

The space $\mathcal{W}$ is equipped with the $1$-norm; thus, for any $w \in \mathcal{W}$ and $W \in \mathscr{B}(\mathcal{W})$, we have
\begin{equation}\label{eq:norm_punctual}
\| w \|_\mathcal{W} \coloneqq \sum_{i = 1}^{2p} | w_i | \qquad \text{ and } \qquad
\| W \|_{\mathscr{B}(\mathcal{W})}\coloneqq \max_{1 \le j \le 2p} \sum_{i = 1}^{2p} | W_{i,j} |.
\end{equation}
%
%
For any $\alpha\in [-1,1]$, consider $w^\star (\alpha) = (x^\star (\alpha), u^\star (\alpha), \beta^\star (\alpha))\in \mathcal{W}$ such that $\mathcal{F} (w^\star (\alpha), \alpha) = 0$.
Thus, each of the two coordinates of $\mathcal{F}$ yields:
(1) $x^\star_0(\kappa)$ is a $p$-periodic point of the map $f$ in~\eqref{eq:sistema_ejemplo1} for $(\beta^\star(\kappa),\kappa)$ with $\kappa\in[\kappa_1,\kappa_2]$, and
(2) its eigenvalue is $\lambda (\kappa) = D_x f^p(x^\star_0 (\kappa), \beta^\star (\kappa), \kappa)= -1$.
Note that this provides a \emph{candidate} for a period-doubling bifurcation, as we do not prove the non-degeneracy condition that guarantees that the eigenvalue crosses -1 monotonically. 

After phrasing the search of period-doubling bifurcations as a zero-finding problem of $\mathcal{F}$, we use a similar strategy of \Cref{thm:contraction} previously done in \Cref{sec:cap} for fixed values of $\alpha$.
Indeed, we numerically compute functions $\alpha \mapsto \bar{w}(\alpha)$ and $\alpha \mapsto \mathcal{A}(\alpha)$, and we construct similar bounds as $Y, Z_1, Z_2$ in~\eqref{eq:bounds}; except now they must be uniform in the $C^0$-norm over $\alpha \in [-1,1]$.
By the uniform contraction mapping theorem, this implies the existence of a $C^0$ function $\alpha \mapsto w^\star(\alpha)$ such that $\mathcal{F}(w(\alpha), \alpha) = 0$ for all $\alpha \in [-1, 1]$.
We state below a uniform contraction version of \Cref{thm:contraction}, which gives us both the existence of the family of solution and an a posteriori error in $C^0$-norm between $\alpha \mapsto \bar{w}(\alpha)$ and $\alpha \mapsto w^\star(\alpha)$.
\begin{theorem}\label{thm:contractionUNIF}
Let $\mathcal{W}$ be a Banach space, $\bar{w} \in C^0([-1,1],\mathcal{W})$, $\mathcal{F} : \mathcal{W} \times [-1,1]\to \mathcal{W}$ a $C^1$ map, 
and $\mathcal{A} \in C^0([-1,1], \mathscr{B}(\mathcal{W}))$ such that $\mathcal{A}(\alpha)$ is injective for all $\alpha \in [-1,1]$.
Fix $R \in [0, \infty]$ and assume the existence of $\mathcal{Y}$, $\mathcal{Z}_1$, and $\mathcal{Z}_2 = \mathcal{Z}_2(R)$ such that
\begin{subequations}\label{eq:bounds_uniform}
\begin{align}
\sup_{\alpha \in [-1,1]} \| \mathcal{A}(\alpha) \mathcal{F}(\bar{w}(\alpha), \alpha) \|_\mathcal{W} &\le \mathcal{Y}, \label{eq:Y_uniform} \\
\sup_{\alpha \in [-1,1]} \| I - \mathcal{A}(\alpha) D_w \mathcal{F}(\bar{w}(\alpha), \alpha) \|_{\mathscr{B}(\mathcal{W})} &\le \mathcal{Z}_1, \label{eq:Z1_uniform} \\
\sup_{\alpha \in [-1,1]} \quad \sup_{\substack{w \in B_R(\bar{w}(\alpha)) \\ w \ne \bar{w}(\alpha)}} \frac{\| \mathcal{A}(\alpha) (D_w \mathcal{F}(w, \alpha) - D_w \mathcal{F}(\bar{w}(\alpha), \alpha)) \|_{\mathscr{B}(\mathcal{W})}}{\| w - \bar{w}(\alpha)\|_\mathcal{W}} &\le \mathcal{Z}_2. \label{eq:Z2_uniform}
\end{align}
\end{subequations}
If there exists $r \in [0, R]$ such that
\begin{equation}\label{eq:condUNIF}
\mathcal{Y} + r (\mathcal{Z}_1 - 1) + \frac{r^2}{2} \mathcal{Z}_2 \le 0
\qquad \text{and} \qquad
\mathcal{Z}_1 + r \mathcal{Z}_2 < 1,
\end{equation}
then there is a unique $w^\star \in C^0([-1,1], B_r (\bar{w}(\alpha)))$ such that $\mathcal{F}(w^\star(\alpha), \alpha) = 0$ for all $\alpha \in [-1,1]$.
\end{theorem}
To apply \Cref{thm:contractionUNIF}, we need to construct the numerical approximations $\alpha \mapsto \bar{w}(\alpha)$ and $\alpha \mapsto \mathcal{A}(\alpha)$, for $\alpha\in [-1,1]$, of the true functions $\alpha \mapsto w^\star(\alpha)$ and $\alpha \mapsto D_w \mathcal{F}(w^\star(\alpha), \alpha)^{-1}$, respectively.
An efficient and well-known numerical method is to use a Chebyshev interpolation~\cite{Trefethen}.
%
Specifically, both $\bar{w}$ and $\mathcal{A}$ are sought as a finite sum of Chebyshev polynomials of the first kind, $T_k(\alpha)$ for each $k\in\mathbb{N}_0$, which are given by the recurrence relation, for all $\alpha \in [-1, 1]$,
\begin{equation}\label{Chebyshev}
T_0(\alpha) = 1, \qquad T_1 (\alpha) = \alpha, \qquad T_k(\alpha) = 2 \alpha T_{k-1}(\alpha) - T_{k-2}(\alpha), \quad k \ge 2.
\end{equation}
%
%
%
In practice, given an order $K \ge 0$, at each fixed Chebyshev node $\alpha^{[k]}\coloneqq -\cos \left(k\pi / K\right)$ with $k\in \{0,\ldots, K\}$, we solve numerically 
\begin{equation}\label{perdob:A}
    \mathcal{F}(\bar{w}^{[k]}, \alpha^{[k]}) \approx 0 \qquad \text{ and } \qquad \mathcal{A}^{[k]} \approx D_w \mathcal{F}(\bar{w}^{[k]}, \alpha^{[k]}) ^{-1}
\end{equation}
for the unknowns $\bar{w}^{[k]}$. 
Then, we efficiently retrieve numerical interpolations for $\alpha \mapsto \bar{w}(\alpha)$ and $\alpha \mapsto \mathcal{A}(\alpha)$ 
using the fast Fourier transform, see~\cite{Trefethen}.

For now we postpone the discussion on estimating the uniform bounds in~\eqref{eq:bounds_uniform} to Step 3 of \Cref{sec:steps} and \Cref{sec:boundsPP}.

\subsection{From period-$2$ to period-$4$}
\label{sec:steps}

\textbf{Step 1 (reformulate the problem).}
We define $\mathcal{F} : \mathcal{W} \times [-1,1]\to \mathcal{W}$, where $\mathcal{W}\coloneqq\mathbb{R}^{2}\times\mathbb{R}\times\mathbb{R}$, as the particular case of \eqref{eq:map_period_doubling} with $p=2$, where a zero of $\mathcal{F}$ corresponds to a continuum of period-2 orbits in which its eigenvalue attains value -1. This yields the map given by~\eqref{eq:map_period_doubling}, with explicit coordinate $F$ and $G$ in~\eqref{def:G}: 
\begin{equation}
F(w, \alpha) =
\begin{pmatrix}
    f (x_1,\beta,\kappa(\alpha)) - x_0 \\
    f (x_0,\beta,\kappa(\alpha)) - x_1 \\
\end{pmatrix}, \qquad 
G(x, u, \beta, \alpha) =
\begin{pmatrix}
    D_x f(x_1, \beta, \kappa(\alpha)) u + 1 \\
    D_x f(x_0, \beta, \kappa(\alpha)) - u
\end{pmatrix}. 
\end{equation}
where $w = (x_0,x_1, u, \beta) \in \mathcal{W} \coloneqq\mathbb{R}^2 \times \mathbb{R}^{1} \times \mathbb{R}$.

\textbf{Step 2 (numerical approximation).}
There are two quantities that are computed numerically: a family $\alpha \mapsto \bar{w}(\alpha)$ such that $\mathcal{F}(\bar{w}(\alpha),\alpha)\approx 0$ and a 
family of Jacobian matrices given by $\alpha \mapsto \mathcal{A}(\alpha) \approx D_w \mathcal{F}(\bar{w}(\alpha), \alpha)^{-1}$.
As mentioned in~\eqref{perdob:A}, we use a Chebyshev interpolation~\cite{Trefethen}: first, for each $k\in \{0,\ldots, K\}$ with $K=2^4$, we find $ \mathcal{F}(\bar{w}^{[k]}, \alpha^{[k]}) \approx 0$ and $\mathcal{A}^{[k]} \approx D_w \mathcal{F}(\bar{w}^{[k]}, \alpha^{[k]}) ^{-1}$, and then we interpolate between these $K=2^4$ nodes, using the fast cosine transform.
Note that the families $\bar{w}(\alpha)$ and $\mathcal{A}(\alpha) $ consist, respectively of a vector and a matrix whose entries are finite sum of Chebyshev polynomials of order $K=2^4$.

\textbf{Step 3 (rigorous -- interval arithmetic).}
We verify the hypothesis of \Cref{thm:contractionUNIF}, where we recall that $\mathcal{W}\coloneqq\mathbb{R}^2\times \mathbb{R}\times \mathbb{R}$ equipped with the $1$-norm described in \eqref{eq:norm_punctual}.
Since the nonlinearity~\eqref{eq:sistema_ejemplo1} is non-polynomial, we need extra care in order to obtain the bounds $\mathcal{Y},\mathcal{Z}_1,\mathcal{Z}_2$ in~\eqref{eq:bounds_uniform} of \Cref{thm:contractionUNIF}.
For our purposes, we use Taylor's theorem to split the nonlinearity into a polynomial term, for which it is easier to directly manipulate Chebyshev polynomials, 
and a remainder term, which we will control directly in $C^0$-norm.

Since a periodic orbit $x^\star = (x_0^\star, x_1^\star)$ has two components which may be far in values from one another, it is advantageous to consider two different Taylor expansions of the nonlinearity nearby different $\chi_0$ and $\chi_1$, e.g. the respective mean value of $\bar{x}_0^{[k]}$ and $\bar{x}_1^{[k]}$ for all $k=0,\ldots,K$:
\begin{equation}
\chi_0 \coloneqq \frac{1}{K+1} \sum_{k = 0}^K \bar{x}_0^{[k]}, \qquad \text{ and } \qquad \chi_1 \coloneqq \frac{1}{K+1} \sum_{k = 0}^K \bar{x}_1^{[k]}.
\end{equation}
Recall that the nonlinearity in~\eqref{eq:sistema_ejemplo1} is given by the reflected sigmoid function, $h(x) \coloneqq - \frac{1}{1 + e^x}$
%
%
%
%
\begin{equation}
h(x) \coloneqq - \frac{1}{1 + e^x}.
\end{equation}
Then, for each $i = 0, 1$, we expand $h(x)$ in Taylor series nearby each $\chi_i$: 
\begin{subequations}
\begin{align}
h(x) &= h_{i, \textnormal{poly}} (x) + \tilde{h}_i(x),\\
h_{i,\textnormal{poly}}(x) &\coloneqq \sum_{n = 0}^N \frac{h^{(n)}(\chi_i)}{n!} (x - \chi_i)^n, \\
\tilde{h}_i (x) &\coloneqq \int_0^1 \frac{(1 - \tau)^N}{N!} h^{(N+1)} (\chi_i + \tau (x - \chi_i)) (x - \chi_i)^{N+1} \, \mathrm{d} \tau, \label{eq:rho}
\end{align}
\end{subequations}
where $h_{i, \textnormal{poly}}(x)$ is the polynomial term, which is truncated after order $N$, and $\tilde{h}_i(x)$ is the remainder; here, we denote $h^{(n)}(x)\coloneqq\frac{\mathrm{d}^n}{\mathrm{d}x^n} h(x)$.

Now, let us define a truncation of our original problem in~\eqref{eq:map_period_doubling} regarding only the polynomial (up to order $N$) term:
\begin{equation}\label{eq:G_poly}
\mathcal{F}_\textnormal{poly}(w, \alpha) \coloneqq
\begin{pmatrix}
F_\textnormal{poly}(x, \beta, \alpha) \\
G_\textnormal{poly}(x, u, \beta, \alpha)
\end{pmatrix}, \qquad w = (x, u, \beta) \in \mathbb{R}^2 \times \mathbb{R}^1 \times \mathbb{R},
\end{equation}
where we use the Taylor expansion of $h$ nearby each $\chi_i$, for $i=1,2$, accordingly:
\begin{subequations}
\begin{align}
F_\textnormal{poly} (x, \beta, \alpha) &\coloneqq 
\begin{pmatrix}
f_{1,\textnormal{poly}}(x_1, \beta, \kappa(\alpha)) - x_0 \\
f_{0,\textnormal{poly}}(x_0, \beta, \kappa(\alpha)) - x_1
\end{pmatrix}, \\
G_\textnormal{poly} (x, u, \beta, \alpha) &\coloneqq 
\begin{pmatrix} 
D_x f_{1,\textnormal{poly}} (x_1, \beta, \kappa(\alpha)) u + 1 \\ D_x f_{0,\textnormal{poly}} (x_0, \beta, \kappa(\alpha)) - u 
\end{pmatrix}.\\
f_{i,\textnormal{poly}}(x, \beta, \alpha) &\coloneqq f(x, \beta, \kappa(\alpha)) - \kappa(\alpha) \tilde{h}_i (x) = \beta + x + \kappa(\alpha) h_{i,\textnormal{poly}} (x), \qquad i = 0, 1,
\end{align}    
\end{subequations}
%
%
Thus, for $w = (x, u, \beta) \in \mathbb{R}^2 \times \mathbb{R} \times \mathbb{R}$ and $\alpha \in [-1,1]$, the remainder of $\mathcal{F}$ is
\begin{equation}\label{eq:G_diff}
\tilde{\mathcal{F}}(w, \alpha) 
\coloneqq \mathcal{F}(w, \alpha) - \mathcal{F}_\textnormal{poly}(w, \alpha)
= \begin{pmatrix*}[l]
\kappa(\alpha) \tilde{h}_1 (x_1) \\
\kappa(\alpha) \tilde{h}_0 (x_0) \\
\kappa(\alpha) \tilde{h}_1^{(1)}(x_1) u \\
\kappa(\alpha) \tilde{h}_0^{(1)}(x_0)
\end{pmatrix*}.
\end{equation}
We have resorted to a Taylor expansion of $h$ for the sake of avoiding too much technicalities, but this is by no means the only way to deal with nonlinear terms.
Indeed, nonlinearities can be handled via automatic differentiation~\cite{Olivier}, or by exploiting the Poisson summation formula \cite[Section 3]{Lindsey}; see also~\cite{church2025globalcont} for another approach dealing with rational nonlinear functions. 
These aforementioned methods provide a much better control on the nonlinearities in contrast to the present uniform $C^0$-norm bound on the Taylor remainder. As a matter of fact, such techniques allow us to set up the zero-finding problem $\mathcal{F}$ as a map defined on the space of sequences of Chebyshev coefficients, $\ell^1$, for which one could apply \Cref{thm:contraction} on the Banach space $\mathcal{X} = \ell^1$;
see~\cite{BredenHenot} for details.

In the next~\Cref{sec:boundsPP}, we compute the bounds in~\eqref{eq:bounds_uniform} and establish the existence of an $r>0$ satisfying~\eqref{eq:condUNIF}, thereby verifying~\Cref{thm:contractionUNIF} for parameters between $\kappa_1=-16$ and $\kappa_2=-13$. We use $K=2^4$ Chebyshev nodes and a Taylor expansion of $h$ of order $N=10$. 
This yields four curves that are period-doubling bifurcations candidates (as non-degeneracy has not been rigorously verified), each with an error of order at most $10^{-4}$.
We then extend these curves by repeating this procedure for $j=1,\ldots,5$, covering the ranges between $\kappa_1=-16-3j$ and $\kappa_2=-13-3j$. 
The resulting curves for $\kappa\in [-31,-13]$ are shown in~\Cref{fig:ec_dp2} (right). 
We chose to patch several parameter windows for efficiency and controlled accuracy, whereas larger parameter intervals could be handled at once with larger $N$ and $K$. 
Note that to glue two consecutive curves, a ball of uniqueness of one endpoint must be within the uniqueness ball of the other, see~\cite{JBLessardMisch10}.
%
Furthermore, numerical evidence indicates that a pair of curves meet at a fold and thus there are only two distinct curves. 
At the fold, the implicit function theorem fails and \Cref{thm:contractionUNIF} can not be applied to our map $\mathcal{F}$. However, the full curve (including the fold) could be obtained via (pseudo-)arclength continuation, which is not pursued here; see \cite{JBLessardMisch10,BredenHenot}.

\subsection{Details for the bounds}\label{sec:boundsPP}

From Step 2, we have obtained $\bar{w}(\alpha)$ and $\mathcal{A}(\alpha)$, for all $\alpha\in [-1,1]$, in terms of Chebyshev polynomials of order $K=2^4$.
To tackle Step 3, we need to manipulate these Chebyshev approximations and estimate the bounds $\mathcal{Y}, \mathcal{Z}_1, \mathcal{Z}_2$ in~\eqref{eq:bounds_uniform}.
%
%
To this end, we represent Chebyshev series by their coefficient sequence, which we view as elements of the Banach space:
\begin{equation}
\ell^1 \coloneqq \left\{
\psi = (\psi_0, \psi_1, \dots) \in \mathbb{R}^{\mathbb{N}_0} \, : \, \| \psi \|_{\ell^1} \coloneqq \sum_{k \in \mathbb{Z}} | \psi_{|k|} | < \infty
\right\}.
\end{equation}
This space is naturally equipped with a multiplication operation $*$ inherited from the product of Chebyshev polynomials, that is, for all $\psi, \phi \in \ell^1$,
\begin{equation}
\psi(\alpha) \phi(\alpha) = \sum_{k \in \mathbb{N}_0} (\psi * \phi)_{k} \, T_k(\alpha), \qquad \text{with} \qquad (\psi * \phi)_k \coloneqq \sum_{j \in \mathbb{Z}} \psi_{|k - j|} \, \phi_{|j|}.
\end{equation}
%
Note that $\ell^1$ is a commutative ring with the addition $+$ and multiplication $*$, and a Banach algebra
\begin{equation}\label{eq:banach_algebra}
\|\psi * \phi \|_{\ell^1} \le \| \psi \|_{\ell^1} \| \phi \|_{\ell^1}, \qquad \text{for all }\psi,\phi\in \ell^1,
\end{equation}
%

%
Recall that any Lipschitz function, i.e. $\psi \in \mathrm{Lip}([-1,1],\mathbb{R})$, 
admits a converging Chebyshev series, see~\cite{Trefethen}. Thus, such a function $\psi$ can be identified with its sequence of Chebyshev coefficients $\psi=(\psi_0,\psi_1,\ldots)\in \ell^1$. 
%
Then, we can control its $C^0$-norm by its $\ell^1$-norm via
\begin{equation}\label{eq:c0_bound}
\sup_{\alpha \in [-1,1]} | \psi(\alpha) | \le \| \psi \|_{\ell^1}.
\end{equation}
%
%
%
Similarly, for vector-valued functions, we identify the space $\mathrm{Lip}([-1,1], \mathcal{W})$, where $\mathcal{W}\coloneqq\mathbb{R}^{p}\times \mathbb{R}^{p-1}\times \mathbb{R}$, with the space $\mathcal{W}(\ell^1)\coloneqq (\ell^1)^{p} \times (\ell^1)^{p-1} \times \ell^1$, so that each coordinate of a Lipschitz function is represented by a Chebyshev series with coefficients in $\ell^1$,
%
endowed with the norm
\begin{equation}\label{eq:norm_uniform}
\| w \|_{\mathcal{W}(\ell^1)} \coloneqq \sum_{i = 1}^{2p} \| w_i \|_{\ell^1}, \qquad \text{for all } w \in \mathcal{W}(\ell^1).
\end{equation}
Thus the analogue to inequality~\eqref{eq:c0_bound} reads
\begin{equation}\label{eq:c0_bound_n1}
\sup_{\alpha \in [-1,1]} \| w(\alpha) \|_\mathcal{W} \le \| w \|_{\mathcal{W}(\ell^1)}.
\end{equation}
The inequality~\eqref{eq:c0_bound_n1} allows us to obtain the bound~~\eqref{eq:Z1_uniform}.
Next, to obtain the bounds in~\eqref{eq:Z1_uniform} and \eqref{eq:Z2_uniform}, we must bound the operator norms a function of operators $\alpha \mapsto W(\alpha) \in \mathscr{B}(\mathcal{W})$ uniformly in $\alpha$.
Upon the identification of $\mathrm{Lip}([-1,1], \mathcal{W})$ with $\mathcal{W}(\ell^1)$, a special type of operators is that of those acting as multiplication operators on $\mathcal{W}(\ell^1)$.
More precisely, for $\psi \in \ell^1$, we denote its multiplication operator by $\mathcal{M}_\psi : \phi \mapsto \psi * \phi$, and any operator $W \in \mathscr{B}(\mathcal{W}(\ell^1))$ acting as a multiplication operator on $\mathcal{W}(\ell^1)$ can be viewed as a matrix $W = (\mathcal{M}_{\omega_{i,j}})_{1 \le i,j \le 2p}$.
It follows that
\begin{equation}\label{eq:norm_uniform2}
\| W \|_{\mathscr{B}(\mathcal{W}(\ell^1))} \leq \max_{1 \le j \le 2p} \sum_{i = 1}^{2p} \| \mathcal{M}_{\omega_{i,j}} \|_{\mathscr{B}(\ell^1)} =
\max_{1 \le j \le 2p} \sum_{i = 1}^{2p} \| \omega_{i,j} \|_{\ell^1},
\end{equation}
and, in particular, we have
\begin{equation}\label{eq:c0_bound_n2}
\sup_{\alpha \in [-1,1]} \| W(\alpha) \|_{\mathscr{B}(\mathcal{W})} \le \| W \|_{\mathscr{B}(\mathcal{W}(\ell^1))}.
\end{equation}
The inequalities \eqref{eq:c0_bound_n2} and \eqref{eq:norm_uniform2} are applicable in our context, as $\mathcal{A}$ and $D_w \mathcal{F}$ are multiplication operators.

\paragraph{\boxed{\text{$\mathcal{Y}$ bound}}}
The triangle inequality implies
\begin{equation}\label{YboundPP}
\sup_{\alpha \in [-1,1]} \| \mathcal{A}(\alpha) \mathcal{F}(\bar{w}(\alpha), \alpha) \|_\mathcal{W} \le  \| \mathcal{A} (\cdot)\mathcal{F}_\textnormal{poly}(\bar{w}(\cdot), \cdot) \|_{\mathcal{W}(\ell^1)} + \| \mathcal{A} (\cdot)\|_{\mathscr{B}(\mathcal{W}(\ell^1))} \sup_{\alpha \in [-1,1]} \| \tilde{\mathcal{F}}(\bar{w}(\alpha), \alpha) \|_\mathcal{W},
\end{equation}
where we used \eqref{eq:c0_bound}.
The first term, which consists of a polynomial nonlinearity, can be rigorously computed using interval arithmetic and \eqref{eq:norm_uniform}.
The second term, which consists of a non-polynomial nonlinearity, can be bounded due to~\eqref{eq:G_diff} as:
\begin{align}
\sup_{\alpha \in [-1,1]} &\|\tilde{\mathcal{F}}(\bar{w}(\alpha), \alpha) \|_\mathcal{W} \nonumber \\
&\le \sup_{\alpha \in [-1,1]} \Big( |\kappa(\alpha) \tilde{h}_1(\bar{x}_1(\alpha))| + |\kappa(\alpha) \tilde{h}_0(\bar{x}_0(\alpha))| + |\kappa(\alpha)  \bar{u}(\alpha) \tilde{h}_1^{(1)}(\bar{x}_1(\alpha)) | + |\kappa(\alpha) \tilde{h}_0^{(1)}(\bar{x}_0(\alpha)) | \Big)\nonumber\\
&\le \|\kappa\|_{\ell^1} \sup_{\alpha \in [-1,1]} \Big(|\tilde{h}_1(\bar{x}_1(\alpha))| + |\tilde{h}_0(\bar{x}_0(\alpha))| + \|\bar{u}\|_{\ell^1} | \tilde{h}_1^{(1)}(\bar{x}_1(\alpha)) | + | \tilde{h}_0^{(1)}(\bar{x}_0(\alpha)) | \Big),
\end{align}
where $\|\kappa\|_{\ell^1}=|\kappa_2+\kappa_1|/2 + |\kappa_2-\kappa_1|/2$.
Moreover, the remainder terms are bounded by
\begin{subequations}\label{Y:cont}
\begin{align}
\sup_{\alpha \in [-1,1]}|\tilde{h}_i (\bar{x}_i(\alpha))|
&\le \frac{\|(\bar{x}_i(\alpha) - \chi_i)^{N+1}\|_{\ell^1}}{(N+1)!} \sup_{\alpha \in [-1,1]} \sup_{\tau \in [0, 1]} | h^{(N+1)} (\chi_i + \tau (\bar{x}_i(\alpha) - \chi_i)) |, \\
\sup_{\alpha \in [-1,1]} |\tilde{h}_i^{(1)} (\bar{x}_i(\alpha))| 
&\le \frac{\|(\bar{x}_i(\alpha) - \chi_i)^N\|_{\ell^1}}{N!} \sup_{\alpha \in [-1,1]} \sup_{\tau \in [0, 1]} | h^{(N+1)} (\chi_i + \tau (\bar{x}_i(\alpha) - \chi_i)) |,
\end{align}
\end{subequations}
%
%
due to~\eqref{eq:rho}.
Hence, all elements in~\eqref{YboundPP} can be computed using interval arithmetic, where for the case $N=10$, we can use the fact that
\begin{equation}\label{der11}
\sup_{x \in \mathbb{R}} |h^{(11)}(x)| = \frac{691}{8}.
\end{equation}

\paragraph{\boxed{\text{$\mathcal{Z}_1$ bound}}}
Once again, 
the triangle inequality implies that
\begin{align}
\sup_{\alpha \in [-1,1]}\| I - \mathcal{A}(\alpha) D_w \mathcal{F}(\bar{w}(\alpha), \alpha) \|_{\mathscr{B}(\mathcal{W})} \nonumber \le 
& \| I - \mathcal{A} (\cdot)D_w \mathcal{F}_\textnormal{poly}(\bar{w}(\cdot), \cdot) \|_{\mathscr{B}(\mathcal{W}(\ell^1))} \\
&+ \| \mathcal{A}(\cdot) \|_{\mathscr{B}(\mathcal{W}(\ell^1))} \sup_{\alpha \in [-1,1]} \| D_w \tilde{\mathcal{F}}(\bar{w}(\alpha), \alpha) \|_{\mathscr{B}(\mathcal{W})},
\end{align}
where we used \eqref{eq:c0_bound_n2}.
The first term, involving only polynomial nonlinearities, is computed via \eqref{eq:norm_uniform2}.
Note that the Jacobian of $\tilde{\mathcal{F}}$ in~\eqref{eq:G_diff}
\begin{align}
D_w \tilde{\mathcal{F}}(w, \alpha) 
=
\begin{pmatrix*}
0 & \kappa(\alpha) \tilde{h}_1^{(1)}(x_1) & 0 & 0 \\
\kappa(\alpha) \tilde{h}_0^{(1)}(x_0) & 0 & 0 & 0\\
0 & \kappa(\alpha) \tilde{h}_1^{(2)}(x_1)u & \alpha(\alpha) \tilde{h}_1^{(1)}(x_1) & 0 \\
\kappa(\alpha) \tilde{h}_0^{(2)}(x_0) & 0 & 0 & 0
\end{pmatrix*}.\label{eq:Ftail_diff}
\end{align}

Due to~\eqref{eq:Ftail_diff}, the operator norm of the remainder is
\begin{align}
\sup_{\alpha \in [-1,1]} &\|D_w \tilde{\mathcal{F}}(\bar{w}(\alpha), \alpha) \|_{\mathscr{B}(\mathcal{W})} \nonumber \\
&= \sup_{\alpha \in [-1,1]} \max\Big(|\kappa(\alpha) \tilde{h}_0^{(1)}(\bar{x}_0(\alpha)) | + |\kappa(\alpha) \tilde{h}_0^{(2)}(\bar{x}_0(\alpha))|, |\kappa(\alpha) \tilde{h}_1^{(1)}(\bar{x}_1 (\alpha)) | + |\kappa(\alpha) \bar{u}(\alpha) \tilde{h}_1^{(2)}(\bar{x}_1 (\alpha))|\Big) \nonumber \\
&\le \|\kappa\|_{\ell^1} \max\Big(\sup_{\alpha \in [-1,1]}\big( |\tilde{h}_0^{(1)}(\bar{x}_0(\alpha)) | + |\tilde{h}_0^{(2)}(\bar{x}_0(\alpha))|\big), \sup_{\alpha \in [-1,1]} \big(|\tilde{h}_1^{(1)}(\bar{x}_1 (\alpha)) | + \|\bar{u}\|_{\ell^1} | \tilde{h}_1^{(2)}(\bar{x}_1 (\alpha))|\big)\Big),
\end{align}
where the following right hand side can also be computed using \eqref{der11} and interval arithmetic:
\begin{equation}
\sup_{\alpha \in [-1,1]} |\tilde{h}_i^{(2)} (\bar{x}_i(\alpha))|
\le \frac{\|(\bar{x}_i (\alpha) - \chi_i)^{N-1}\|_{\ell^1}}{(N-1)!} \sup_{\alpha \in [-1,1]}\, \sup_{\tau\in[0,1]} | h^{(N+1)} (\chi_i + \tau (\bar{x}_i (\alpha) - \chi_i)) |.
\end{equation}

\paragraph{\boxed{\text{$Z_2$ bound}}}
Pick $R >0$.
We need to find the upper bound
\begin{equation}
\mathcal{Z}_2 \ge \sup_{\alpha \in [-1,1]} \, \sup_{\substack{w \in B_R(\bar{w}(\alpha)) \\ w \ne \bar{w}(\alpha)}} \frac{\| \mathcal{A}(\alpha) (D_w \mathcal{F}(w, \alpha) - D_w \mathcal{F}(\bar{w}(\alpha), \alpha)) \|_{\mathscr{B}(\mathcal{W})}}{\| w - \bar{w}(\alpha) \|_\mathcal{W}}.
\end{equation}
Using the mean value theorem, we obtain
\begin{equation}
\mathcal{Z}_2 \ge \| \mathcal{A}(\cdot) \|_{\mathscr{B}(\mathcal{W}(\ell^1))} \sup_{\alpha \in[-1,1]} \, \sup_{w \in B_R(\bar{w}(\alpha))} \|D_w^2\mathcal{F}(w, \alpha)\|_{\mathscr{B}(\mathcal{W}, \mathscr{B}(\mathcal{W}))}.
\end{equation}
Here $\mathscr{B}(\mathcal{W}, \mathscr{B}(\mathcal{W}))$ denotes the space of bi-linear forms.
A straightforward computation yields
\begin{align}
\sup_{\alpha \in[-1,1]} \, &\sup_{w \in B_R(\bar{w}(\alpha))} \|D_w^2\mathcal{F} (w, \alpha)\|_{\mathscr{B}(\mathcal{W}, \mathscr{B}(\mathcal{W}))} \nonumber \\
&= \sup_{\alpha \in[-1,1]} \, \sup_{w \in B_R(\bar{w}(\alpha))} \, \max_{1\le j\le 4} \sum_{1\le l \le 4} \max_{1\le i\le 4} | \partial_{w_i} \partial_{w_j} \mathcal{F}_l (w, \alpha) | \nonumber \\
&\le \|\kappa\|_{\ell^1} \sup_{x \in\mathbb{R}} \, \max \Big(
|h^{(2)}(x)| + |h^{(3)}(x)|,
|h^{(2)}(x)| + \max\big(|h^{(2)}(x)|, (\| \bar{u}\|_{\ell^1} + R) |h^{(3)}(x)| \big)
\Big) \nonumber \\
&= \|\kappa\|_{\ell^1} \left(\frac{\sqrt{3}}{18} + \max\Big(\frac{\sqrt{3}}{18}, \max\big(1,\| \bar{u}\|_{\ell^1} + R\big) \frac{1}{8} \Big) \right),
\end{align}
where $\partial_{w_i}$ stands for the partial derivative with respect to the $i$-th component of $w \in \mathcal{W} = \mathbb{R}^4$, and $\mathcal{F}_l$ for the $l$-th component of $\mathcal{F}$.
Moreover, we used the fact that
\begin{equation}
\sup_{x \in \mathbb{R}} |h^{(2)}(x)| = \frac{\sqrt{3}}{18}, \qquad \sup_{x \in \mathbb{R}} |h^{(3)}(x)| = \frac{1}{8}.
\end{equation}
%


\textbf{Acknowledgments. }
We thank Jos\'e M. Arrieta for bringing us together, without whom these pages would not exist, and for several stimulating discussions.
P. Lappicy was supported by Marie Sklodowska--Curie Actions, UNA4CAREER H2020 Cofund, 847635, with the project DYNCOSMOS.

\bibliography{library}


\end{document}